\documentclass[11pt, a4paper]{amsart}

\usepackage{amsmath,amsthm,amssymb}
\usepackage[colorlinks,urlcolor=blue,linkcolor=blue,citecolor=blue]{hyperref}
\usepackage{tikz} 

\newcommand\nnode[1]{{\kern -0.6pt\mathop\bigcirc\limits_{\rlap{#1}}\kern -1pt}}
\newcommand\edge{{\vrule width20pt height3pt depth-2pt}}
\newcommand\halfedge{{\vrule width8pt height3pt depth-2pt}}
\newcommand\vertbar[2]{\rlap{\kern4pt\vrule width1pt height17.3pt depth-7.3pt}
 \rlap{\raise19.4pt\hbox{$\kern -0.4pt\bigcirc#1$}}{#2}}

\newtheorem{theorem}{Theorem}[section]
\newtheorem{lemma}[theorem]{Lemma}

\newtheorem{proposition}[theorem]{Proposition}
\newtheorem{corollary}[theorem]{Corollary}

\theoremstyle{definition}
\newtheorem{rem}[theorem]{Remark}
\newtheorem{rems}[theorem]{Remarks}
\newtheorem{Qu}[theorem]{Question}

\newcommand\BZ{{\mathbb Z}}
\newcommand\BF{{\mathbb F}}

\newcommand\FS{{\mathfrak S}}
\newcommand\bw{{\mathbf w}}

\newcommand\Torelli{{\mathcal S}{\mathcal I}}
\newcommand\Burau{{\mathcal B}{\mathcal I}}

\newcommand\Id{\mathrm{Id}}

\newcommand\Ker{\mathrm{Ker}}
\newcommand\card{\mathrm{card}}

\newcommand\Sp{\operatorname{Sp}}
\newcommand\St{\operatorname{St}}

\newcommand\LHS{\operatorname{LHS}}
\newcommand\RHS{\operatorname{RHS}}

\newcommand\inv{^{-1}}

\numberwithin{equation}{section}

\title[Braid groups and symplectic Steinberg groups]{Braid groups and\\ symplectic Steinberg groups}

\author{Fran\c cois Digne}
\address{Fran\c cois Digne: 
Laboratoire Ami\'enois de Math\'ematique Fondamentale et Appliqu\'ee,
Universit\'e de Picardie Jules-Verne \& CNRS,
UFR Sciences, 33 Rue Saint-Leu, 80039 Amiens Cedex 01, France}
\email{digne@u-picardie.fr}
\urladdr{www.lamfa.u-picardie.fr/digne/}

\author{Christian Kassel}
\address{Christian Kassel: 
Institut de Recherche Math\'e\-ma\-tique Avanc\'ee,
Universit\'e de Strasbourg \& CNRS,
7 rue Ren\'{e} Descartes, 67084 Strasbourg Cedex, France}
\email{kassel@math.unistra.fr}
\urladdr{irma.math.unistra.fr/\raise-2pt\hbox{\~{}}kassel/}



\begin{document}

\begin{abstract}
We construct a homomorphism~$f$ from the braid group $B_{2n+2}$ on $2n+2$ strands
to the Steinberg group~$\St(C_n,\BZ)$ associated with the Lie type~$C_n$ and with integer coefficients. 
This homomorphism lifts the well-known symplectic representation (aka the integral Burau representation)
of the braid group. We also describe the image and the kernel of~$f$.
\end{abstract}

\maketitle

\section{Introduction}\label{sec-intro}

In this article we provide a connection between low-dimensional topology and algebraic $K$-theory. 
More precisely, let $B_{2n+2}$ be the braid group on $2n+2$ strands ($n\geq 2$). 
In~\cite{Ka1}, following work by Arnold, Magnus \& Peluso, Birman, A'Campo \emph{et al}.\
(see \cite{Ac, Ar, Bi1, Bi2, MP}), 
the second-named author investigated an action of~$B_{2n+2}$ on the free group~$F_{2n}$ on $2n$ generators
obtained by viewing a twice-punctured surface of genus~$n$ as a double covering of the disk
via an hyperelliptic involution. 
Line\-ar\-izing this action, one obtains a homomorphism $\bar{f} : B_{2n+2} \to \Sp_{2n}(\BZ)$
from the braid group to the symplectic modular group~$\Sp_{2n}(\BZ)$.
This symplectic representation was shown by Gambaudo and Ghys 
to be the Burau representation specialized at $t= -1$ (see \cite[Prop.~2.1]{GG}).

Now $\Sp_{2n}(\BZ)$ is a Chevalley group of Lie type~$C_n$. By~\cite{StM,St0,St} it has a natural extension, 
its Steinberg group~$\St(C_n,\BZ)$, 
which is defined by means of a presentation by generators and relations.
(Steinberg groups are basic ingredients in algebraic $K$-theory; see for instance~\cite{Mi}.)

Our main observation is that the symplectic representation~$\bar{f}$ (aka the integral Burau representation mentioned above)
can be lifted to a homomorphism 
\[
f: B_{2n+2} \to \St(C_n,\BZ)
\]
from the braid group to the Steinberg group. 

When $n\geq 3$ the lifting~$f$ is not surjective; we shall describe its image. 
We also show how to extend~$f$ to a surjective homomorphism with a bigger Artin group as domain.

We further describe the kernel of~$f$.
As an application we obtain a simple braid-like presentation of the image of~$f$ (resp.\ of the image of~$\bar{f}$),
which is a subgroup of finite index of~$\St(C_n,\BZ)$ (resp.\ of~$\Sp_{2n}(\BZ)$).

The paper is a continuation of~\cite{Ka2}, which dealt with the case $n=2$.
It is organized as follows. 
In Section~\ref{sec-Steinberg} we give a presentation of the Steinberg group~$\St(C_n,\BZ)$ 
and list a few properties of the special elements~$w_{\gamma}$. 
In Section~\ref{sec-BSt} we construct the lifting~$f$ from the braid group to the Steinberg group.
In Section~\ref{sec-image} we determine its image (see Theorem~\ref{thm-image}).
Section~\ref{sec-kern} is devoted to a description of the kernel of~$f$: 
we highlight a braid $\alpha_n \in B_{2n+2}$ which together with two other braids generate the kernel of~$f$ as a normal subgroup
(see Theorem~\ref{thm-kernel}).
In Section~\ref{sec-epi} 
we extend~$f$ to an epimorphism $\widehat{f} :  \widehat{B}_{2n+2} \to \St(C_n,\BZ)$,
where $\widehat{B}_{2n+2}$ is an Artin group (see Theorem~\ref{thm-St-newgen})
slightly bigger than the braid group~$B_{2n+2}$.

\section{The Steinberg group~$\St(C_n,\BZ)$}\label{sec-Steinberg}

With any irreducible root system~$\Phi$ Steinberg~\cite{St0,St} associated the so-called Steinberg group,
which is an extension of the simple complex algebraic group of type~$\Phi$. 
Later Stein~\cite{StM} extended Steinberg's construction over any commutative ring~$R$, 
thus leading to the Steinberg group~$\St(\Phi,R)$.
We are interested in the case when the root system~$\Phi$ is of type~$C_n$ ($n\geq 2$) and $R = \BZ$ is the ring
of integers. The corresponding Steinberg group~$\St(C_n,\BZ)$ is an extension of the symplectic modular group~$\Sp_{2n}(\BZ)$.

\subsection{The symplectic modular group $\Sp_{2n}(\BZ)$}\label{ssec-Sp}

Let $n$ be an integer $\geq 2$.
Recall that $\Sp_{2n}(\BZ)$ is the group of $2n\times 2n$ matrices $M$ with integral entries satisfying 
the relation $M^TJ_{2n}M=J_{2n}$, where $M^T$ is the transpose of~$M$,
\[
J_{2n} =
\begin{pmatrix}
0&\Id_n\cr
-\Id_n&0
\end{pmatrix}
\]
and $\Id_n$ is the identity matrix of size~$n$.

Let $\{\varepsilon_1, \ldots, \varepsilon_n\}$ be the canonical basis of the free abelian group~$\BZ^n$.
In Bourbaki's notation the root system of type $C_n$ consists of the following elements (see~\cite[Chap.~VI, \S~4.6]{Bo}):
the elements $\pm\varepsilon_i\pm\varepsilon_j$ ($1\leq i,j\leq n$, $i\neq j$) are the \emph{short roots} and 
$\pm2\varepsilon_i$ ($1\leq i\leq n$) are the \emph{long roots}.

Denote by $E_{i,j}$ the $2n\times 2n$ matrix which has all
entries equal to~$0$ except the $(i,j)$-entry which is equal to~$1$. 
The group $\Sp_{2n}(\BZ)$ is generated by the following matrices (see \cite{Br}):
\begin{itemize}
\item $X_{i,j}=\Id_{2n}+E_{i,j}-E_{j+n,i+n}$ for $1\leq i,j\leq n$, $i\neq j$,
\item $Y_{i,j}=\Id_{2n}+E_{i,j+n}+E_{j,i+n}$ for $1\leq i, j\leq n$, $i\neq j$,
\item $Y'_{i,j}=Y_{i,j}^T$ for $1\leq i, j\leq n$, $i\neq j$,
\item $Z_i=\Id_{2n}+E_{i,i+n}$ for $1\leq i\leq n$,
\item $Z'_i=Z_i^T$ for $1\leq i\leq n$.
\end{itemize}
Note that  $Y_{i,j}=Y_{j,i}$ and $Y'_{i,j}=Y'_{j,i}$.

Each of these matrices generates a root subgroup corresponding to a root in the following way: 
$X_{i,j}$ corresponds to the root $\varepsilon_i-\varepsilon_j$,
$Y_{i,j}$ to the root $\varepsilon_i+\varepsilon_j$, 
$Y'_{i,j}$ to $-\varepsilon_i-\varepsilon_j$,
$Z_i$ to~$2\varepsilon_i$, and $Z'_i$ to~$-2\varepsilon_i$.

We now list the commutation relations between pairs of these matrices corresponding to non-opposite roots. 
In the following relations, the indices $i$, $j$ and $k$  are pairwise
distinct and run over $\{1,\ldots,n\}$:
\begin{equation*}
[X_{i,j},X_{j,k}] = X_{i,k} \, , \quad
[X_{i,j},Y_{j,k}] = Y_{i,k} \, , \quad
\end{equation*}
\begin{equation*}
[X_{i,j},Y'_{i,k}] = Y_{j,k}^{\prime-1} \, , \quad
[Y_{i,j},Y'_{j,k}] = X_{i,k} \, ,
\end{equation*}
\begin{equation*}
[X_{i,j},Y_{i,j}] = Z_i^2 \, , \quad
[X_{i,j},Y'_{i,j}] = Z_j^{\prime -2} \, , 
\end{equation*}
\begin{equation*}
[X_{i,j},Z_j] = Z_iY_{i,j} = Y_{i,j}Z_i \, , \quad
[X_{i,j},Z'_i] = Z'_jY_{i,j}^{\prime-1} = Y_{i,j}^{\prime-1}Z'_j \, ,
\end{equation*}
\begin{equation*}
[Y_{i,j},Z'_i] = X_{j,i}Z_j^{-1} = Z_j^{-1} X_{j,i} \, , \quad
[Y'_{i,j},Z_i] = X_{i,j}^{-1} Z_j^{\prime -1} = Z_j^{\prime -1} X_{i,j}^{-1} \, .
\end{equation*}
The matrices commute for all other pairs of generators, except for
$(X_{i,j}, X_{j,i})$, $(Y_{i,j}, Y'_{i,j})$ and $(Z_i, Z'_i)$, which are pairs corresponding to opposite roots.

\subsection{A presentation of the Steinberg group}\label{ssec-St}

By~\cite[Sect.\,3]{Be2} and~\cite{StM} the Steinberg group $\St(C_n,\BZ)$ has a presentation 
with the same generators and relations as above, namely with generators 
$x_{i,j}$, $y_{i,j}$, $y'_{i,j}$ ($1\leq i,j\leq n$ and $i\neq j$),
$z_i$, $z'_i$ ($1\leq i\leq n$)
subject to the following relations (where $i, j, k \in \{1,\ldots,n\}$ are pairwise distinct):
\begin{equation}
y_{i,j} = y_{j,i}\, , \qquad y'_{i,j} = y'_{j,i} \, ,
\end{equation}
\begin{equation}\label{eq-xxx}
[x_{i,j},x_{j,k}] = x_{i,k}\, ,
\end{equation}
\begin{equation}\label{eq-xyy}
[x_{i,j},y_{j,k}] = y_{i,k} \, ,
\end{equation}
\begin{equation}\label{eq-xyprime}
[x_{i,j},y'_{i,k}] = y_{j,k}^{\prime-1}\, ,
\end{equation}
\begin{equation}\label{eq-yyprime}
[y_{i,j},y'_{j,k}] = x_{i,k}\, ,
\end{equation}
\begin{equation}\label{eq-xyz}
[x_{i,j},y_{i,j}] = z_i^2 \, ,
\end{equation}
\begin{equation}
[x_{i,j},y'_{i,j}] = z_j^{\prime -2} \, ,
\end{equation}
\begin{equation}\label{eq-xz}
[x_{i,j},z_j] = z_i y_{i,j} = y_{i,j} z_i \, ,
\end{equation}
\begin{equation}\label{eq-St8}
[x_{i,j},z'_i] = z'_j y_{i,j}^{\prime-1} = y_{i,j}^{\prime-1} z'_j \, ,
\end{equation}
\begin{equation}\label{eq-St9}
[y_{i,j},z'_i] = x_{j,i} z_j^{-1} = z_j^{-1} x_{j,i} \, ,
\end{equation}
\begin{equation}\label{eq-St10}
[y'_{i,j},z_i] = x_{i,j}^{-1} z_j^{\prime -1} = z_j^{\prime -1} x_{i,j}^{-1} \, ,
\end{equation}
and all remaining pairs of generators commuting, except the pairs
$(x_{i,j}, x_{j,i})$, $(y_{i,j}, y'_{i,j})$ and $(z_i, z'_i)$ for which we do not prescribe any relation.

Note that in view of \eqref{eq-St8} and \eqref{eq-St9} the generators $x_{i,j}$ and $y'_{i,j}$ can be expressed 
in terms of the other generators.

By construction there is a surjective homomorphism 
\[
\pi : \St(C_n,\BZ) \to \Sp_{2n}(\BZ)
\]
sending each generator of $\St(C_n,\BZ)$ represented by a lower-case letter to the symplectic matrix
represented by the corresponding upper-case letter.

By~\cite[Th.\,6.3]{Ma} and~\cite[Kor.\,3.2]{Be2} the kernel of the epimorphism $\pi : \St(C_n,\BZ) \to \Sp_{2n}(\BZ)$
is infinite cyclic generated by $(x_{2\varepsilon_i} \, x_{-2\varepsilon_i}^{-1} \, x_{2\varepsilon_i})^4$,
where $x_{2\varepsilon_i}$ (resp.\ $x_{-2\varepsilon_i}$) is the generator corresponding to the long root~$2\varepsilon_i$
(resp.\ to~$-2\varepsilon_i$); this generator is independent of~$i$ and central (see Lemma~\ref{lem-w4} below;
see also~\cite{La}).

\subsection{The elements $w_{\gamma}$}\label{ssec-w}

For a root $\gamma$ let $x_{\gamma}$ be the generator of the Steinberg group corresponding to~$\gamma$.
Set 
\begin{equation}\label{def-w}
w_{\gamma} = x_{\gamma} \, x_{-\gamma}^{-1} \, x_{\gamma} \in \St(C_n,\BZ) \, .
\end{equation}
In particular, we have
\begin{equation}\label{def-wz}
w_{2\varepsilon_i} = z_i \, z_i^{\prime -1} \, z_i 
\quad\text{and}\quad
w_{-2\varepsilon_i} =  z'_i \, z_i^{-1} \, z'_i \, .
\end{equation}
For simplicity we write $w_i$ for~$w_{2\varepsilon_i}$. 
Since $z_i$ commutes with $z_j$ and with $z'_j$ when $i\neq j$,
we have $w_i w_j = w_j w_i$ for all $(i,j )\in \{1, \ldots, n\}^2$.

The following equality holds for all roots~$\gamma$:
\begin{equation}\label{eq-ww}
w_{\gamma} = w_{-\gamma}^{-1}
\end{equation}
(for a proof, see~\cite[Lemma\,2.2]{Ka2}).
As a consequence, we have
\begin{equation}\label{eq-wxw1}
w_{\gamma} x_{\gamma} w_{\gamma}^{-1}
= w_{-\gamma}^{-1} x_{\gamma} w_{\gamma}^{-1}
= x_{-\gamma}^{-1} \, x_{\gamma} \, x_{-\gamma}^{-1} x_{\gamma} x_{\gamma}^{-1} \, x_{-\gamma} \, x_{\gamma} ^{-1}
= x_{-\gamma}^{-1} .
\end{equation}
Similarly, 
\begin{equation}\label{eq-wxw2}
w_{\gamma} x_{-\gamma} w_{\gamma}^{-1} = x_{\gamma}^{-1} .
\end{equation}
It follows from \eqref{eq-wxw1} and \eqref{eq-wxw2} that the square~$w_{\gamma}^2$ commutes with~$x_{\gamma}$
and with~$x_{-\gamma}$ for all roots~$\gamma$.
 
We also need the subsequent relation between an element~$w_{\gamma}$ and the generator~$x_{\delta}$
associated with a root~$\delta$ such that $\gamma + \delta \neq 0$, namely
\begin{equation}\label{wxw}
w_{\gamma} \, x_{\delta} \, w_{\gamma}^{-1} = x_{\delta'}^{c} \, ,
\end{equation}
where $\delta'$ is the image of~$\delta$ under the reflection~$s_{\gamma}$
in the hyperplane orthogonal to~$\gamma$ and $c = \pm 1$ 
(see Relation\,(R7) in~\cite[Chap.\,3, p.\,23]{St}). 
Recall that~$s_{\gamma}$ is given by
\begin{equation*}
s_{\gamma}(\delta) = \delta - 2 \frac{(\gamma,\delta)}{(\gamma,\gamma)} \gamma\, ,
\end{equation*}
where $( -, -)$ is the inner product of the Euclidean vector space of which 
the set $\{\varepsilon_1, \ldots, \varepsilon_n\}$ forms an orthonormal basis.
To determine the sign~$c$ (and the root~$\delta'$) in~\eqref{wxw} it is enough to  compute the image 
$\pi(w_{\gamma} \, x_{\delta} \, w_{\gamma}^{-1})$ in~$\Sp_{2n}(\BZ)$.

In particular, for any long root~$2\varepsilon_i$ the element $w_i = w_{2\varepsilon_i}$ commutes with all
generators $x_{k,\ell}$, $y_{k,\ell}$ and $y'_{k,\ell}$ such that $k\neq i \neq \ell$.
By contrast we have the non-trivial relations ($i\neq j$)
\begin{equation}\label{eq-wyw}
w_i \, y_{i,j} \, w_i^{-1} = x_{j,i} \, ,
\qquad 
w_i \, x_{j,i} \, w_i^{-1} = y_{i,j}^{-1} \, ,
\end{equation}
and
\begin{equation}\label{eq-wyw2}
w_i \, y'_{i,j} \, w_i^{-1} = x_{i,j} \, ,
\qquad 
w_i \, x_{i,j} \, w_i^{-1} = y_{i,j}^{\prime -1} 
\end{equation}
in the Steinberg group~$\St(C_n,\BZ)$.
Hence, the conjugation by the square~$w_i^2$ turns each generator $x_{i,j}$, $x_{j,i}$, $y_{i,j}$, $y'_{i,j}$ into its inverse,
namely
\begin{equation}\label{eq-wwyww}
w_i^2 \, y_{i,j} \, w_i^{-2} = y_{i,j}^{-1} \, ,
\qquad 
w_i^2 \, x_{j,i} \, w_i^{-2} = x_{j,i}^{-1} \, ,
\end{equation}
and
\begin{equation}\label{eq-wwyww2}
w_i^2 \, y'_{i,j} \, w_i^{-2} = y_{i,j}^{\prime -1} \, ,
\qquad 
w_i^2 \, x_{i,j} \, w_i^{-2} = x_{i,j}^{-1} \, .
\end{equation}

The following lemma will be used in the sequel.

\begin{lemma}\label{lem-w4}
For each $i = 1, \ldots, n$
the element~$w_i^4$ is central in~$ \St(C_n,\BZ)$ and we have $w_i^4 = w_1^4$.
\end{lemma}

\begin{proof}
The centrality of~$w_i^4$ follows from \eqref{eq-wxw1}, \eqref{eq-wxw2}, \eqref{eq-wwyww} and \eqref{eq-wwyww2}.
Now all elements~$w_i$ are conjugate as a consequence of the following special cases of~\eqref{wxw}, where $i\neq j$:
\begin{equation*}
w_{\varepsilon_i-\varepsilon_j} \, z_i\, w_{\varepsilon_i-\varepsilon_j}^{-1} = z_j 
\quad\text{and}\quad
w_{\varepsilon_i-\varepsilon_j} \, z'_i\, w_{\varepsilon_i-\varepsilon_j}^{-1} = z'_j \, .
\end{equation*}
The conclusion follows.
\end{proof}

\section{From the braid group to the Steinberg group}\label{sec-BSt}

Let $B_{2n+2}$ be the braid group on $2n+2$ strands, where $n \geq 2$ is a fixed integer. 
It has a standard presentation with $2n+1$ generators 
$\sigma_1, \sigma_2,\ldots, \sigma_{2n+1}$
and the following relations ($1 \leq i,j \leq 2n+1$):
\begin{equation}\label{braid2}
\sigma_i \sigma_j \sigma_i  = \sigma_j \sigma_i \sigma_j \quad \text{if} \; |i-j| = 1 ,
\end{equation}
and
\begin{equation}\label{braid1}
\sigma_i \sigma_j  = \sigma_j  \sigma_i  \qquad\;\; \text{otherwise}.
\end{equation}

Let us now construct a homomorphism from~$B_{2n+2}$ to the symplectic Steinberg group~$\St(C_n,\BZ)$.

\begin{theorem}\label{thm-BtoSt}
There exists a homomorphism 
$f: B_{2n+2} \to \St(C_n,\BZ)$ such that 
\begin{equation*}
f(\sigma_1) = z_1 \, , \qquad 
f(\sigma_{2n+1}) = z_n \, ,
\end{equation*}
\begin{equation*}
f(\sigma_{2i}) = z_i^{\prime-1} \qquad \text{for} \;\, i=1,\ldots,n,
\end{equation*}
\begin{equation*}
f(\sigma_{2i+1}) = z_i z_{i+1} y_{i,i+1}^{-1} \qquad \text{for} \;\, i=1,\ldots,n-1.
\end{equation*}
The homomorphism~$f$ is surjective if and only if $n=2$.
\end{theorem}

\begin{rems}\label{rems-f}
(a) By~\cite{Ka1} the homomorphism $\bar{f} : B_{2n+2} \to \Sp_{2n}(\BZ)$ mentioned in the introduction is 
defined on the generators~$\sigma_i$ by
\begin{equation*}
\bar{f}(\sigma_1) = Z_1 \, , \qquad 
\bar{f}(\sigma_{2n+1}) = Z_n \, ,
\end{equation*}
\begin{equation*}
\bar{f}(\sigma_{2i}) = Z_i^{\prime-1} \qquad \text{for} \;\, i=1,\ldots,n,
\end{equation*}
\begin{equation*}
\bar{f}(\sigma_{2i+1}) = Z_i Z_{i+1} Y_{i,i+1}^{-1} \quad \text{for} \;\, i=1,\ldots,n-1,
\end{equation*}
where $Y_{i,i+1}$, $Z_i$ and $Z'_i$ are the symplectic matrices defined in Section\,\ref{ssec-Sp}.
It follows from these formulas and from the definition of~$f$ in Theorem\,\ref{thm-BtoSt}
that the latter is a natural lifting of~$\bar{f}$, i.e.\ we have $\bar{f} = \pi \circ f$,
where $\pi: \St(C_n,\BZ) \to \Sp_{2n}(\BZ)$ is the natural projection.

(b) By definition of~$f$ and of $w_1 = w_{2\varepsilon_1}$ (see \eqref{def-wz}) we have
\begin{equation}\label{def-w1}
w_1 = f(\sigma_1 \sigma_2 \sigma_1) \, .
\end{equation}
It follows from this equality and the remark at the end of Section\,\ref{ssec-St}
that the kernel of $\pi : \St(C_n,\BZ) \to \Sp_{2n}(\BZ)$, 
which is generated by~$w_1^4$, belongs to the image of~$f$.
We thus recover the known fact that the integral Burau representation is not injective.

(c) Note that all three factors in the product $z_i z_{i+1} y_{i,i+1}^{-1}$ expressing $f(\sigma_{2i+1})$ commute.
\end{rems}

\begin{proof}
It suffices to check that the values of $f(\sigma_i)$ ($1 \leq i \leq 2n+1$) 
in the Steinberg group~$\St(C_n,\BZ)$ satisfy the braid relations~\eqref{braid2} and~\eqref{braid1}.

(i) Let us first check the trivial commutation relations~\eqref{braid1}.

\begin{itemize}
\item
\emph{Commutation of~$f(\sigma_1)$ with $f(\sigma_{2n+1})$.}
This follows from the fact that  $z_1$ and $z_n$ commute.

\item
\emph{Commutation of~$f(\sigma_1)$ with $f(\sigma_{2i})$ when $i\geq 2$.}
Indeed, $z_1$ commutes with $z'_i$ when $i \neq 1$.

\item
\emph{Commutation of~$f(\sigma_1)$ with all $f(\sigma_{2i+1})$.}
This is implied by the commuting of~$z_1$ with the other $z_i$ and with the generators $y_{i,j}$.
\end{itemize}

Similarly for the trivial braid relations involving $f(\sigma_{2n+1}) = z_n$.

\begin{itemize}
\item
\emph{Commutation of~$f(\sigma_{2i})$ with $f(\sigma_{2j})$.}
This follows from the fact that the generators $z'_i$ commute with one another.

\item
\emph{Commutation of~$f(\sigma_{2i+1})$ with $f(\sigma_{2j+1})$.} Indeed, the $z_i$'s commute with one another,
as do the $y_{i,j}$'s. Moreover, the $z_i$'s commute with the $y_{i,j}$'s.

\item
\emph{Commutation of~$f(\sigma_{2i})$ with $f(\sigma_{2j+1})$ when $i\notin \{i-1,i\}$}.
This is implied by the facts that $z'_i$ commutes with $y_{j,j+1}$ when $i\neq j$ and that
$z'_i$ commutes with $z_j z_{j+1}$ when $j \notin \{i-1,i\}$.
\end{itemize}

(ii) The relation $f(\sigma_1) f(\sigma_2) f(\sigma_1) = f(\sigma_2) f(\sigma_1) f(\sigma_2)$ reads as
\begin{equation*}
z_1 z_1^{\prime-1} z_1 = z_1^{\prime-1} z_1 z_1^{\prime-1} \, ,
\end{equation*}
which is equivalent to $w_{2\varepsilon_1}  = w_{-2\varepsilon_1}^{-1}$, 
where we use the notation of Section~\ref{ssec-w}. The latter equality holds by~\eqref{eq-ww}.

(iii) The relation $f(\sigma_{2i}) f(\sigma_{2i+1}) f(\sigma_{2i}) = f(\sigma_{2i+1}) f(\sigma_{2i}) f(\sigma_{2i+1})$ 
reads for $1\leq i \leq n-1$ as
\begin{equation*}
z_i^{\prime-1}  z_i z_{i+1} y_{i,i+1}^{-1} z_i^{\prime-1}   = 
z_i z_{i+1} y_{i,i+1}^{-1} z_i^{\prime-1} z_i z_{i+1} y_{i,i+1}^{-1} \, .
\end{equation*}
Let $\LHS$ (resp.\ $\RHS$) be the element of~$\St(C_n,\BZ)$ represented by the left-hand (resp.\ right-hand) side of the 
previous equation.

Since $z_i$ commutes with $y_{i,i+1}$ and with $z_{i+1}$, and the latter with~$z'_i$, we have
\begin{equation*}
\LHS = z_{i+1} z_i^{\prime-1} y_{i,i+1}^{-1} z_i z_i^{\prime-1} \, .
\end{equation*}
By \eqref{def-wz}, \eqref{eq-St9}, \eqref{eq-ww} and the trivial commutation relations we obtain
\begin{eqnarray*}
\LHS & = & z_{i+1} z_i^{\prime-1} y_{i,i+1}^{-1} z'_i w_{-2\varepsilon_i}^{-1} \\
& = & z_{i+1} y_{i,i+1}^{-1} z_i^{\prime-1}  z_{i+1} x_{i+1,i}^{-1} z'_i w_i \\
& = & z_{i+1} y_{i,i+1}^{-1}   z_{i+1} x_{i+1,i}^{-1}  w_i \\
& = & z_{i+1}^2 y_{i,i+1}^{-1}  x_{i+1,i}^{-1}  w_i \, .
\end{eqnarray*}

Let us now deal with $\RHS$. 
Since $z_i$ commutes with $y_{i,i+1}$, and $z_{i+1}$ with $z_i$, $z'_i$ and $y_{i,i+1}$, we have
\begin{eqnarray*}
\RHS &= & z_{i+1}^2   y_{i,i+1}^{-1} z_i z_i^{\prime-1} z_i  y_{i,i+1}^{-1}
= z_{i+1}^2   y_{i,i+1}^{-1} w_i y_{i,i+1}^{-1} \\
&= & z_{i+1}^2   y_{i,i+1}^{-1} x_{i+1,i}^{-1} w_i 
= \LHS.
\end{eqnarray*}
For the third equality we have used~\eqref{eq-wyw}.

(iv) Applying the automorphism $\sigma_i \mapsto \sigma_{2n+2-i}$ of~$B_{2n+2}$, we reduce
the relations 
$f(\sigma_{2i-1}) f(\sigma_{2i}) f(\sigma_{2i-1}) = f(\sigma_{2i}) f(\sigma_{2i-1}) f(\sigma_{2i})$
($2 \leq i \leq n$) and 
$f(\sigma_{2n}) f(\sigma_{2n+1}) f(\sigma_{2n}) = f(\sigma_{2n+1}) f(\sigma_{2n}) f(\sigma_{2n+1})$
to the previous cases.

(v) The surjectivity for $n=2$ was established in~\cite{Ka2}. 
Let us now prove that $f$ is not surjective when $n\geq 3$.
We remark that under the composition $B_{2n+2} \rightarrow \Sp_{2n}(\BZ) \rightarrow \Sp_{2n}(\BF_2)$
of $\bar f$ with the reduction modulo~$2$
the images of the generators $\sigma_i$ of $B_{n+2}$ have order~$2$. Hence
this morphism factors through the symmetric group ${\FS}_{2n+2}$ of all permutations of the set~$\{1, \ldots, 2n+2\}$. 
If $f$ is surjective, then $\bar f$ is surjective too and we obtain a surjective morphism
${\FS}_{2n+2} \rightarrow \Sp_{2n}(\BF_2)$ since the reduction modulo~$2$ is surjective. 
Now for $n\geq 2$ the symmetric group~${\FS}_{2n+2}$ has no non-commutative proper quotient, 
which implies that the map ${\FS}_{2n+2} \rightarrow \Sp_{2n}(\BF_2)$ is an isomorphism. But this is impossible
as $\Sp_{2n}(\BF_2)$ is a simple group when $n \geq 3$ and ${\FS}_{2n+2}$ is not.
(See also \cite[Proof of Statement~B]{Ar}.)
\end{proof}

\section{The image of the homomorphism~$f$}\label{sec-image}

As noted in Theorem~\ref{thm-BtoSt}, the homomorphism $f: B_{2n+2} \to \St(C_n,\BZ)$
is not surjective when $n \geq 3$. We can nevertheless determine its image~$f(B_{2n+2})$ 
inside~$\St(C_n,\BZ)$.

Consider the \emph{level~$2$ congruence subgroup} $\Sp_{2n}(\BZ)[2]$ defined as the kernel of the homomorphism
$\Sp_{2n}(\BZ) \rightarrow \Sp_{2n}(\BF_2)$
induced by reduction modulo~$2$. We lift $\Sp_{2n}(\BZ)[2]$ to the Steinberg group by taking its preimage
\begin{equation}\label{def-St[2]}
\St(C_n,\BZ)[2] = \pi^{-1} \left( \Sp_{2n}(\BZ)[2] \right)
\end{equation}
under the canonical projection $\pi: \St(C_n,\BZ) \to \Sp_{2n}(\BZ)$. 
Thus $\St(C_n,\BZ)[2]$ is the kernel of the composition
\[
\St(C_n,\BZ) \overset{\pi}{\longrightarrow} \Sp_{2n}(\BZ) \longrightarrow \Sp_{2n}(\BF_2).
\]
The group $\Sp_{2n}(\BF_2)$ being finite, $\St(C_n,\BZ)[2]$ is of finite index in~$\St(C_n,\BZ)$.

Since by~\cite[Kor.\,3.2]{Be2} the kernel of the projection $\pi : \St(C_n,\BZ) \to \Sp_{2n}(\BZ)$ 
is the infinite cyclic group $\langle w_1^4 \rangle$ generated by~$w_1^4$, we have the short exact sequence
\begin{equation*}
1 \to \langle w_1^4 \rangle \longrightarrow \St(C_n,\BZ)[2] \overset{\pi}{\longrightarrow} \Sp_{2n}(\BZ)[2] \to 1.
\end{equation*}

Recall the surjective homomorphism $p: B_{2n+2} \to \FS_{2n+2}$ sending each generator~$\sigma_i$ of~$B_{2n+2}$ 
to the simple transposition~$s_i \in \FS_{2n+2}$, where $s_i$ permutes $i$ and~$i+1$ and leaves the remaining
elements of~$\{1, \ldots, 2n+2\}$ fixed. The kernel of~$p$ is the \emph{pure braid group}~$P_{2n+2}$.

\begin{theorem}\label{thm-image}
Assume $n\geq 2$.
(a) The images under~$f$ of the pure braid groups $P_{2n+2}$  and $P_{2n+1}$ 
are both equal to~$\St(C_n,\BZ)[2]$:
\begin{equation*}
f(P_{2n+2}) = f(P_{2n+1}) = \St(C_n,\BZ)[2].
\end{equation*}

(b) The image of the full braid group $B_{2n+2}$ fits into the short exact sequence
\begin{equation*}
1 \to \St(C_n,\BZ)[2] \longrightarrow f(B_{2n+2}) \longrightarrow \FS_{2n+2} \to 1.
\end{equation*}
\end{theorem}

When $n=2$ the homomorphism $f$ is surjective, that is $f(B_6) = \St(C_2,\BZ)$. 
One recovers in this way the well-known isomorphism $\Sp_4(\BF_2) \cong {\FS}_6$.

In general $f(B_{2n+2})$ is of finite index in~$\St(C_n,\BZ)$ with index $i_n$ equal to
\begin{equation*}
i_n =  \frac{\card\,\Sp_{2n}(\BF_2)}{\card\,\FS_{2n+2}} 
= 2^{n^2} \frac{\prod_{i=1}^n(2^{2i}-1)}{(2n+2)!}\
\end{equation*}
(see \cite[p.\,64]{Ch}).
The values of~$i_n$ grow very rapidly; for low subscripts they are: 
$i_2 = 1$, $i_3 = 36$, $i_4 = 13\,056$, $i_5 = 51\,806\,208$, 
$i_6 = 2\,387\,230\,064\,640$.

\begin{proof}
By Remark\,\ref{rems-f}\,(b) the kernel $\langle w_1^4 \rangle$ of $\pi: \St(C_n,\BZ) \to \Sp_{2n}(\BZ)$
is in the image of~$f$. 
Therefore, in view of the definitions of~$\St(C_n,\BZ)[2]$ and of~$P_{2n+2}$
it is enough to check that the image of $\bar{f} = \pi \circ f : B_{2n+2} \to \Sp_{2n}(\BZ)$ satisfies the following two properties: 
(i) $\bar{f}(P_{2n+2}) = \bar{f}(P_{2n+1}) = \Sp_{2n}(\BZ)[2]$ and (ii) there is an isomorphism
$\bar{f}(B_{2n+2})/\Sp_{2n}(\BZ)[2] \cong \FS_{2n+2}$. 

As we observed in Part\,(v) of the proof of Theorem\,\ref{thm-BtoSt},
the composition $B_{2n+2} \overset{\bar{f}}{\longrightarrow} \Sp_{2n}(\BZ) \longrightarrow \Sp_{2n}(\BF_2)$
factors through ${\FS}_{2n+2}$. 
Hence we have the inclusions $\bar{f}(P_{2n+1}) \subset \bar{f}(P_{2n+2}) \subset \Sp_{2n}(\BZ)[2]$. 

The induced morphism ${\FS}_{2n+2} \rightarrow \Sp_{2n}(\BF_2)$ 
(which is non-trivial since the image of $\sigma_1$ in $\Sp_{2n}(\BF_2)$ is not the identity) 
is injective since the images of $\sigma_1$ and $\sigma_2$ do not commute
and ${\FS}_{2n+2}$ has no non-commutative proper quotient when $n\geq 2$.
This implies that Property\,(ii) above follows from~(i). Let us now prove the latter.

In order to establish the opposite inclusion $\bar{f}(P_{2n+1}) \supset \Sp_{2n}(\BZ)[2]$, 
it is sufficient to prove $\bar{f}(B_{2n+1}) \supset\Sp_{2n}(\BZ)[2]$.
We appeal to~\cite{Ac}. In \emph{loc.\ cit}.\ A'Campo considers the monodromy representation of~$B_{2n+1}$ 
in the free $\BZ$-module~$V$ with basis $\{\delta_i, \, i=1 ,\ldots,2n\}$, 
endowed with an alternating form $I$ satisfying $I(\delta_i,\delta_{i+1})=1$ for $i=1,\ldots,2n-1$. 
This form is non-degenerate.
The monodromy representation maps each generator~$\sigma_i$ of~$B_{2n+1}$ 
to the automorphism $T_i$ of~$V$ defined by
\begin{equation*}
T_i(\delta_j) =
\begin{cases}
\delta_j& \text{ if } j\neq i-1,i+1,\\
\delta_j+\delta_i &\text{ if }j=i-1,\\
\delta_j-\delta_i &\text{ if }j=i+1.
\end{cases}
\end{equation*}
One thus obtains a representation of~$B_{2n+1}$ in the symplectic group $\Sp_{2n}(\BZ)$.

Now by~\cite[Th.\,1\,(2)]{Ac} the image of this monodromy representation contains the congruence subgroup~$\Sp_{2n}(\BZ)[2]$.
To conclude, it suffices to check that this representation is isomorphic to our~$\bar f$. Indeed, 
If we take 
\[
\left\{ \delta_1,\delta_1+\delta_3,\ldots,\delta_1+\delta_3+\cdots+\delta_{2n-1}, \delta_2,\delta_4,\ldots,\delta_{2n} \right\}
\]
as a basis of~$V$, the matrix of the form~$I$ becomes in $V \cong \BZ^{2n}$ 
the matrix~$J_{2n}$ considered in \S\,\ref{ssec-Sp}. 
One easily checks that in this new basis the matrix of~$T_i$ is equal to that of~$\bar f(\sigma_i)$ for $i=1,\ldots,2n$. 
\end{proof}

\begin{rem}
As noted by a referee, Conditions (i)--(ii) above may also be derived from recent results of
Bloomquist, Patzt and Scherich (see~\cite{BPS}).
\end{rem}

\begin{rem}
The squares of all the generators $x_{i,j}$, $y_{i,j}$, $y'_{i,j}$, $z_i$, $z'_i$
of the Steinberg group~$\St(C_n,\BZ)$ lie in the image of $f$, 
and even in the image of the pure braid group~$P_{2n+2}$ since they are in~$\St(C_n, \BZ)[2]$.
\end{rem}

\begin{rem}
Given a prime~$p$, the level~$p$ congruence subgroup $\Sp_{2n}(\BZ)[p]$ is
defined as the kernel of the homomorphism $\Sp_{2n}(\BZ) \rightarrow \Sp_{2n}(\BF_p)$ induced by reduction modulo~$p$.
By~\cite[Prop.\,6.7]{FM} it is torsion-free when $p\neq 2$.
Let $\St(C_n,\BZ)[p]$ be the kernel of the composite map
$\St(C_n,\BZ) \overset{\pi}{\longrightarrow} \Sp_{2n}(\BZ) \rightarrow \Sp_{2n}(\BF_p)$;
it is of finite index in~$\St(C_n,\BZ)$.
Now $\St(C_n,\BZ)[p]$ is an extension of the torsion-free group $\Sp_{2n}(\BZ)[p]$ by the infinite cyclic group $\langle w_1^4 \rangle$;
therefore it is torsion-free as well ($p\neq 2$).
It follows that the Steinberg group $\St(C_n,\BZ)$ is \emph{virtually torsion-free}, 
i.e.\ contains a finite-index torsion-free subgroup.
\end{rem}

\section{The kernel of~$f$}\label{sec-kern}

In this section we determine the kernels of $\bar{f}: B_{2n+2} \to \Sp_{2n}(\BZ)$ and of $f: B_{2n+2} \to \St(C_n,\BZ)$.

\subsection{The elements~$\Delta_k$ and their images}\label{ssec-Delta}

We start with a few simple-looking elements of the kernel of~$f$.

For $k\in \{2, \ldots, 2n+2\}$
let $\Delta_k \in B_k$ be the longest reduced positive braid in the group~$B_k$ of braids with $k$ strands.
It is defined inductively by $\Delta_2 = \sigma_1$ and 
$\Delta_{k+1} = (\sigma_1 \sigma_2 \cdots \sigma_k) \, \Delta_k$ for $k\geq 2$.
It is well known (see \cite{Bi} or \cite[Sect.\,1.3.3]{KT}) that its square $\Delta_k^2$ generates the center of~$B_k$.

We now consider $\Delta_k$ as an element of~$B_{2n+2}$ under the natural inclusion $B_k \to B_{2n+2}$
sending each generator $\sigma_i \in B_k$ ($1 \leq i \leq k-1$) to $\sigma_i \in B_{2n+2}$.
We may thus look for the image of~$\Delta_k$ under the homomorphism $f: B_{2n+2} \to \St(C_n,\BZ)$. 

Let us first compute the image $f(\Delta_k^2)$ of the squares of some of the braids~$\Delta_k$. 

\begin{proposition}\label{prop-fDelta2}
Let $f: B_{2n+2} \to \St(C_n,\BZ)$ be the homomorphism defined in Theorem\,\ref{thm-BtoSt}.
For each $i \in \{1, \ldots, n\}$ we have
\begin{equation*}
f(\Delta_{2i+1}^2) = \left(w_1^4 \right)^{i(i-1)/2} \prod_{j=1}^i\, w_j^2 \,.
\end{equation*}
Moreover,
\begin{equation*}
f(\Delta_{2n+2}^2) = \left(w_1^4 \right)^{n(n+1)/2} .
\end{equation*}
\end{proposition}

\begin{rem}
The element $\Delta_{2n+2}^2$ generating the center of~$B_{2n+2}$ belongs to the kernel of~$\bar f$;
it does not belong to the kernel of~$f$ 
since the order of~$w_1^4$ is infinite in the Steinberg group~$\St(C_n,\BZ)$ (see~\cite[Th.\,6.3]{Ma}).
\end{rem}

The following consequence of Proposition~\ref{prop-fDelta2} provides us with 
non-trivial elements of the kernel of~$f$.

\begin{corollary}\label{coro-fDelta4}
(a) If $1 \leq i \leq n$, then
\begin{equation*}
f \left(\Delta_{2i+1}^4  \, \Delta_3^{-4i^2} \right) = 1.
\end{equation*}
(b) We also have
\begin{equation*}
f \left( \Delta_{2n+2}^2 \,\Delta_3^{-2n(n+1)} \right) = 1.
\end{equation*}
\end{corollary}

\begin{proof}
By Proposition~\ref{prop-fDelta2} and Lemma~\ref{lem-w4} we have
$f \left(\Delta_{2i+1}^4 \right) = \left(w_1^4 \right)^{i^2}$. In the special case $i=1$
we obtain $f(\Delta_3^4)= w_1^4$. Therefore
\begin{equation*}
f \left(\Delta_{2i+1}^4 \right) = \left(w_1^4 \right)^{i^2} = f \left( \Delta_3^{4i^2} \right) .
\end{equation*}
Similarly, $f \left( \Delta_{2n+2}^2 \right) = f ( \Delta_3^{2n(n+1)} )$.
\end{proof}

To prove Proposition\,\ref{prop-fDelta2} we need two preliminary lemmas.

\begin{lemma}\label{lem-fsigma-odd}
We have 
\begin{equation*}
f(\sigma_1 \sigma_2 \cdots\sigma_{2i}) =
\begin{cases}
w_1z_1\inv & \text{ if } i = 1, \\
\noalign{\medskip}
\left( \prod_{j=1}^{i} \, w_j \right) 
\left( x_{1,2} x_{2,3} \cdots x_{i-1,i} \right) z_i\inv &\text{ if } 1 < i \leq n,
\end{cases}
\end{equation*}
and
\begin{equation*}
f(\sigma_1 \sigma_2 \cdots\sigma_{2i+1}) =
\begin{cases}
z_1 &\text{ if } i=0, \\
\noalign{\medskip}
w_1y_{1,2}\inv z_2 &\text{ if } i=1, \\
\noalign{\medskip}
\left( \prod_{j=1}^{i} \, w_j \right) \left( x_{1,2} x_{2,3} \cdots x_{i-1,i} \right) y_{i,i+1}\inv z_{i+1} &\text{ if } 1 < i<n,\\
\noalign{\medskip}
\left( \prod_{j=1}^{n}\, w_j \right) \left( x_{1,2} x_{2,3} \cdots x_{n-1,n} \right) &\text{ if } i=n.
\end{cases}
\end{equation*}
\end{lemma}

\begin{proof}
We proceed by induction on the length $k$ of the braid word $\sigma_1 \cdots \sigma_k$. 
For $k = 1,2$ we have 
$f(\sigma_1) = z_1$ and $f(\sigma_1 \sigma_2)  = z_1 z_1^{\prime -1} = w_1 z_1\inv$ by definition of~$f$
and of~$w_1$.

For $k\geq 3$, using the induction and~\eqref{eq-wyw}, we obtain
\begin{eqnarray*}
f(\sigma_1\cdots\sigma_{2i})
& = & f(\sigma_1\cdots\sigma_{2i-1}) f(\sigma_{2i}) \\
& = & \left( \prod_{j=1}^{i-1} \, w_j \right) \left( x_{1,2} x_{2,3} \cdots x_{i-2,i-1} \right) y_{i-1,i}\inv z_i z_i^{\prime -1} \\
& = & \left( \prod_{j=1}^{i-1} \, w_j \right) \left( x_{1,2} x_{2,3} \cdots x_{i-2,i-1} \right) y_{i-1,i}\inv w_i z_i\inv \\
& = & \left( \prod_{j=1}^{i-1} \, w_j \right) \left( x_{1,2} x_{2,3} \cdots x_{i-2,i-1} \right)  w_i x_{i-1,i} z_i\inv .
\end{eqnarray*}
Observing that $w_i$ commutes with the $x_{j,j+1}$'s to its left, we obtain the desired formula.

To obtain the formula for $f(\sigma_1\cdots\sigma_{2i+1})$ it suffices to multiply the formula for $f(\sigma_1\cdots\sigma_{2i})$
on the right by $f(\sigma_{2i+1})$ which is equal to $z_i\inv z_i y_{i,i+1}\inv z_{i+1}$ if $1 < i < n$
and to~$z_n$ if $i=n$ and to cancel the product $z_i\inv z_i$.
\end{proof}

A similar computation yields the following result.

\begin{lemma}\label{lem-fsigma-even}
We have 
\begin{equation*}
f(\sigma_{2i}\cdots\sigma_2\sigma_1) =
\begin{cases}
z_1\inv w_1 & \text{ if } i = 1, \\
\noalign{\medskip}
z_i^{\prime-1}z_i \left(\prod_{j=1}^{i-1} \, w_j \right) \left( x_{i,i-1} \cdots x_{3,2} x_{2,1} \right)
&\text{ if } 1 < i \leq n,
\end{cases}
\end{equation*}
and
\begin{equation*}
f(\sigma_{2i+1}\cdots\sigma_2\sigma_1)=
\begin{cases}
z_{i+1} \left( \prod_{j=1}^{i} \, w_j \right)  \left( x_{i+1,i} x_{i,i-1} \cdots x_{3,2} x_{2,1} \right) &\text{ if } 1 \leq i < n,\\
\noalign{\medskip}
\left( \prod_{j=1}^{n} \, w_j \right)  \left( x_{n,n-1} x_{n-1,n-2} \cdots x_{3,2} x_{2,1} \right) &\text{ if } i=n.
\end{cases}
\end{equation*}
\end{lemma}

We can now prove Proposition\,\ref{prop-fDelta2}.

\begin{proof}[Proof of Proposition\,\ref{prop-fDelta2}]
(a) For the first equality we argue by induction on~$i$. 
For $i=1$ we have 
$f(\Delta_3^2) = f(\sigma_1\sigma_2\sigma_1)^2 = w_1^2$.
The induction formula 
\begin{equation}\label{eq-induction-Delta2}
\Delta^2_{k+1} = \Delta^2_k \, (\sigma_k \sigma_{k-1} \cdots \sigma_2\sigma_1)
(\sigma_1\sigma_2 \cdots \sigma_{k-1}\sigma_k)
\end{equation}
allows us to obtain $f(\Delta^2_{2i+1})$ from $f(\Delta^2_{2i-1})$ by multiplying the latter on the right by
$f(\sigma_{2i-1}\cdots\sigma_1\sigma_1\cdots\sigma_{2i-1}) f(\sigma_{2i}\cdots\sigma_1\sigma_1\cdots\sigma_{2i})$.

We first compute $f(\sigma_{2i-1}\cdots\sigma_1\sigma_1\cdots\sigma_{2i-1})$.
By Lemmas~\ref{lem-fsigma-odd} and~\ref{lem-fsigma-even},
\begin{multline*}
f(\sigma_{2i-1}\cdots\sigma_1\sigma_1\cdots\sigma_{2i-1}) \\
= z_i\left ( \prod_{j=1}^{i-1}w_j \right )
x_{i,i-1} x_{i-1,i-2} \cdots x_{3,2}x_{2,1}
\left ( \prod_{j=1}^{i-1}w_j \right )
x_{1,2}x_{2,3}\cdots x_{i-2,i-1}y_{i-1,i}\inv z_i \, .
\end{multline*}
Using~\eqref{eq-wyw}, we have
$x_{k+1,k} w_k w_{k+1} = w_k w_{k+1} x_{k,k+1}^{-1}$ for all $k= 1, \ldots, i-1$.
These relations allow us to push all the $w_j$'s to the left, yielding
\begin{eqnarray*}
&& \hskip -40pt
f(\sigma_{2i-1}\cdots\sigma_1\sigma_1\cdots\sigma_{2i-1}) = \\
& = & z_i \left ( \prod_{j=1}^{i-2}w_j^2 \right ) w_{i-1} x_{i,i-1} w_{i-1} \times \\
&& \times 
\left( x_{i-2,i-1}\inv\cdots x_{2,3}\inv x_{1,2}\inv x_{1,2}x_{2,3}\cdots x_{i-2,i-1} \right) y_{i-1,i}\inv z_i \\
& = & z_i \left ( \prod_{j=1}^{i-1}w_j^2 \right )  (y_{i-1,i} y_{i-1,i}\inv) z_i 
= z_i \left ( \prod_{j=1}^{i-1}w_j^2 \right ) z_i 
= \left ( \prod_{j=1}^{i-1}w_j^2 \right ) z_i^2
\end{eqnarray*}
since $z_i$ commutes with the other~$w_j$'s.

In the same way we obtain
\begin{eqnarray*}
&& \hskip -30pt
f(\sigma_{2i}\ldots\sigma_1\sigma_1\ldots\sigma_{2i})=\\
& = & z_i^{\prime-1}z_i \left ( \prod_{j=1}^{i-1}w_j \right )
x_{i,i-1} \cdots x_{3,2}x_{2,1} \left ( \prod_{j=1}^i w_j \right )
x_{1,2} \cdots x_{i-1,i} z_i^{-1} \\
& = & z_i^{\prime-1}z_i \left ( \prod_{j=1}^{i-1}w_j^2 \right ) w_i 
\left( x_{i-1,i}\inv \cdots x_{2,3}\inv x_{1,2}\inv x_{1,2} \cdots x_{i-1,i} \right) z_i^{-1} \\
& = & z_i^{\prime-1} z_i z'_i \left ( \prod_{j=1}^{i-1}w_j^2 \right ) w_i \, ,
\end{eqnarray*}
the last equality resulting from~\eqref{eq-wxw1} and the fact that $z'_i$ commutes with $w_j$ when $j\neq i$.

Using the previous calculations, Equation~\eqref{eq-wxw2} and Lemma~\ref{lem-w4}, we deduce 
\begin{eqnarray*}
f(\Delta_{2i+1}^2)
& = & f(\Delta_{2i-1}^2) f(\sigma_{2i-1}\cdots\sigma_1\sigma_1\cdots\sigma_{2i-1}) 
f(\sigma_{2i}\cdots\sigma_1\sigma_1\cdots\sigma_{2i}) \\
&=& f(\Delta_{2i-1}^2)\left ( \prod_{j=1}^{i-1}w_j^2 \right )
z_i^2z_i^{\prime-1}z_i z'_i  \left ( \prod_{j=1}^{i-1}w_j^2\right ) w_i \\
&= & f(\Delta_{2i-1}^2) \left ( \prod_{j=1}^{i-1}w_j^4 \right ) z_i w_i z'_i w_i  
= f(\Delta_{2i-1}^2) \left ( \prod_{j=1}^{i-1}w_j^4 \right ) z_i  z_i\inv w_i^2  \\
& = & f(\Delta_{2i-1}^2) \left ( \prod_{j=1}^{i-1}w_j^4 \right ) w_i^2 
= f(\Delta_{2i-1}^2) \, w_1^{4(i-1)}  w_i^2 \, .
\end{eqnarray*}
Using the induction hypothesis, we obtain
\begin{equation*}
f(\Delta_{2i+1}^2)
= (w_1^4)^{(i-1)(i-2)/2} \left (\prod_{j=1}^{i-1}\, w_j^2 \right ) w_1^{4(i-1)} w_i^2
= (w_1^4)^{i(i-1)/2} \prod_{j=1}^{i}\, w_j^2 \, , 
\end{equation*}
which is the desired formula.

(b) We now prove the second formula of Proposition~\ref{prop-fDelta2}. 
By Lemmas~\ref{lem-fsigma-odd} and~\ref{lem-fsigma-even} we have
\begin{multline*}
f(\sigma_{2n+1} \cdots\sigma_1\sigma_1 \cdots \sigma_{2n+1}) \\
= \left( \prod_{j=1}^{n} \, w_i \right )\left( x_{n,n-1} \cdots x_{2,1} \right )
\left( \prod_{j=1}^{n} \, w_i \right) \left( x_{1,2} x_{2,3} \cdots x_{n-1,n} \right)  .
\end{multline*}
Pushing the $w_i$'s to the left as above, we obtain
\begin{equation*}
f(\sigma_{2n+1} \cdots\sigma_1\sigma_1 \cdots \sigma_{2n+1})
=  \prod_{j=1}^{i=n} \, w_j^2 \, .
\end{equation*}
Hence, by the induction formula, the first formula of the proposition for the case $i=n$,
and Lemma~\ref{lem-w4} we have
\begin{eqnarray*}
f(\Delta_{2n+2}^2) 
& = & f(\Delta_{2n+1}^2) f(\sigma_{2n+1} \cdots\sigma_1\sigma_1 \cdots \sigma_{2n+1}) \\
& = & (w_1^4)^{n(n-1)/2} \left(\prod_{j=1}^n\, w_j^2\right)^2 
= (w_1^4)^{n(n-1)/2} \prod_{j=1}^n\, w_j^4 \\
& = & (w_1^4)^{n(n-1)/2} (w_1^4)^{n} 
= (w_1^4)^{n(n+1)/2} .
\end{eqnarray*}
\end{proof}

\subsection{A special element of the kernel of~$f$}\label{ssec-abab}

Recall from\,\cite{Ka2} that for $n=2$ the kernel of~$f: B_6 \to \St(C_2,\BZ)$ is the normal closure of the braid
\begin{equation}\label{kerB6}
(\sigma_1 \sigma_2 \sigma_1)^2 (\sigma_1 \sigma_3^{-1} \sigma_5) (\sigma_1 \sigma_2 \sigma_1)^{-2} 
(\sigma_1 \sigma_3^{-1} \sigma_5)\, .
\end{equation}
We now exhibit a similar element of the kernel of~$f: B_{2n+2} \to \St(C_n,\BZ)$ when $n\geq 3$. 

For any $n\geq2$ set
\begin{equation}\label{def-beta}
\beta_n = \prod_{i=0}^{n} \, \sigma_{2i+1}^{(-1)^i} 
= \sigma_1 \sigma_3^{-1} \cdots \sigma_{2n+1}^{(-1)^n} \in B_{2n+2} \, .
\end{equation}
For $n= 2$, the braid $\beta_2$ is the element $\sigma_1 \sigma_3^{-1} \sigma_5$ appearing in~\eqref{kerB6}.
Since by Relations\,\eqref{braid2} the odd-numbered generators $ \sigma_{2i+1}$ commute with one another,
the product in\,\eqref{def-beta} can be taken in any order.
Observe also that the square $\beta_n^2$ of~$\beta_n$, being the product of squares of generators of the braid group, 
belongs to the pure braid group~$P_{2n+2}$.

As the $z_j$'s commute with one another and with the $y_{i,i+1}$'s, for all $n\geq 2$ we obtain
\begin{equation}\label{eq-fbeta}
f(\beta_n) = \prod_{i=1}^{n-1} y_{i,i+1}^{(-1)^{i+1}} 
= y_{1,2} \, y_{2,3}^{-1}  \cdots y_{n-1,n}^{(-1)^n} \, .
\end{equation}

Using the elements~$\Delta_k$ introduced in Section\,\ref{ssec-Delta}, we set
\begin{equation}\label{def-gamma}
\gamma_n = \prod_{i=1}^{n-1} \, \Delta_{2i+1}^2 = \Delta_3^2 \Delta_5^2 \cdots \Delta_{2n-1}^2 \in B_{2n+2} \, .
\end{equation}
We have $\gamma_2 = \Delta_3^2 = (\sigma_1\sigma_2\sigma_1)^2$. 
It follows from the centrality of each~$\Delta_k^2$ in the braid group~$B_k$ 
that the product defining~$\gamma_n$ can be taken in any order. 
Clearly, $\gamma_n$ is a pure braid for all $n \geq 2$. 
Actually, $\gamma_n \in P_{2n-1} \subset P_{2n+2}$.

By analogy with~\eqref{kerB6} we consider the element 
\begin{equation}\label{def-alpha}
\alpha_n = \gamma_n \beta_n \gamma_n^{-1} \beta_n \in B_{2n+2} \, .
\end{equation}
Note that $\alpha_n$ belongs to the pure braid group~$P_{2n+2}$ as it can be expressed as the product 
$\alpha_n = \gamma_n (\beta_n \gamma_n^{-1} \beta_n^{-1}) \beta_n^2$ of pure braids. 
Moreover, $\alpha_n$ is non-trivial since it is the product of $(\sigma_{2n+1}^2)^{(-1)^n}$ with an element of~$B_{2n}$.
Clearly, $\alpha_2$ is the braid appearing in~\eqref{kerB6}.

\begin{proposition}\label{prop-ker}
For all $n\geq 2$ we have $f(\alpha_n) = 1$ in the Steinberg group~$\St(C_n,\BZ)$.
\end{proposition}

\begin{proof}
Since the conjugation by the central element~$w_1^4$ is trivial, it follows from Proposition~\ref{prop-fDelta2}
that the conjugation by~$f(\Delta_{2i+1}^2)$ is equal to the conjugation by~$w_1^2 w_2^2 \cdots w_i^2$.
Therefore, the conjugation by~$f(\gamma_n)$ is equal to the conjugation by~$\prod_{k=1}^{n-1} w_k^{2(n-k)}$.
In the previous product we may omit each factor whose exponent $2(n-k)$ is divisible by~$4$.
Thus the conjugation by~$f(\gamma_n)$ is equal to the conjugation by~$w_{n-1}^2 w_{n-3}^2 w_{n-5}^2\cdots$. 
Now the conjugation by $w_{n-2k+1}^2$ is non-trivial only on the factors $y_{n-2k,n-2k+1}$ and $y_{n-2k+1,n-2k+2}$ of~$f(\beta_n)$
and turns each of them into its inverse.
Therefore, $f(\gamma_n) f(\beta_n) f(\gamma_n)^{-1}  = f(\beta_n)^{-1}$,
from which one deduces the desired result.
\end{proof}

\begin{Qu}
Is there a geometric interpretation for the braid~$\alpha_n \in B_{2n+2}$, 
for instance in terms of Dehn twists?
\end{Qu}

\subsection{The kernels of~$\bar{f}$ and of~$f$}\label{ssec-kernel}

We now state our main result on these kernels.

In case $n=2$ the kernel $\Ker(f)$ is the normal closure of~$\alpha_2$
and $\Ker(\bar{f})$ is the normal closure of~$\alpha_2$  and 
$(\sigma_1 \sigma_2 \sigma_1)^4 = (\sigma_1 \sigma_2)^6 = \Delta_3^4$ (see~\cite[Th.\, 4.1 and Cor.\,4.2]{Ka2}).

We now turn to the general case.
Note that by Part\,(v) of the proof of Theorem\,\ref{thm-BtoSt},
the kernel of $\bar f$, hence also the kernel of $f$,  is contained in the pure braid group~$P_{2n+2}$.

\begin{theorem}\label{thm-kernel}
Assume $n\geq 3$.

(a) The kernel of $\bar{f}: B_{2n+2} \to \Sp_{2n}(\BZ)$ is the normal closure of the set consisting of the three braids
$\Delta_3^4$, $\Delta_5^4$ and $\alpha_n$.

(b) The kernel of $f: B_{2n+2} \to \St(C_n,\BZ)$ is the normal closure of the set consisting of
the commutator $[\sigma_3, \Delta_3^4]$ and of the elements $\Delta_5^4  \, \Delta_3^{-16}$ and $\alpha_n$.
\end{theorem}

As a consequence, the finite-index subgroup $\bar{f}(B_{2n+2})$ of~$\Sp_{2n}(\BZ)$ has a presentation 
with $2n+1$ generators $\sigma_1, \sigma_2,\ldots, \sigma_{2n+1}$ 
subject to the braid relations \eqref{braid2}, \eqref{braid1} and the three additional relations
\begin{equation}\label{rel-alpha}
\Delta_3^4 = \Delta_5^4  = \alpha_n = 1.
\end{equation}

Similarly, $f(B_{2n+2})$ has a presentation with the same generators subject to the braid relations and the relations
$[\sigma_3, \Delta_3^4] = \Delta_5^4  \, \Delta_3^{-16} = \alpha_n = 1$.

\begin{proof}
(a) As we observed in the proof of Theorem~\ref{thm-image}, 
the restriction of~$\bar{f}$ to the subgroup~$B_{2n+1}$ of~$B_{2n+2}$ is the monodromy representation considered in~\cite{Ac}. 
The kernel of this restriction is the \emph{hyperelliptic Torelli group} $\Torelli_n^1$ investigated in~\cite{BMP}.
In \emph{loc.\ cit}.\ Brendle, Margalit, Putman prove that $\Torelli_n^1$ is isomorphic to a subgroup~$\Burau_{2n+1}$
of the mapping class group of the disk with $2n+1$ marked points. 
By Theorem~C of~\cite{BMP} and the comments thereafter the subgroup~$\Burau_{2n+1}$
is generated by squares of Dehn twists about curves in the disk surrounding exactly $3$ or~$5$ marked points. 
Now under the standard identification of the mapping class group with the 
braid group~$B_{2n+1}$ (see e.g.\,\cite[\S\,1.6]{KT}), the square of a Dehn twist about a curve surrounding $3$
(resp.\ $5$) marked points corresponds in the braid group to a conjugate of the element~$\Delta_3^4$ (resp.\ of~$\Delta_5^4$). 
Thus, $\Ker(\bar{f}) \cap B_{2n+1} = \Ker(\bar{f}) \cap P_{2n+1}$ is the normal closure of $\Delta_3^4$ and~$\Delta_5^4$.

To determine the whole kernel inside~$P_{2n+2}$, we use the fact that 
$P_{2n+2}$ is the semi-direct product of a normal free group and of~$P_{2n+1}$ (see \cite[\S\,1.3]{KT}).
This free group has $2n+1$~generators $A_{i,2n+2}$ ($i= 1, \ldots, 2n+1$) defined by $A_{2n+1,2n+2} = \sigma_{2n+1}^2$
and
\begin{equation}\label{def-A}
A_{i,2n+2} = (\sigma_{2n+1} \sigma_{2n} \cdots \sigma_{i+1}) \, \sigma_i^2 \,
(\sigma_{2n+1} \sigma_{2n} \cdots \sigma_{i+1})^{-1}
\end{equation}
when $1 \leq i \leq 2n$. The following expression for~$A_{i,2n+2}$ will be used in the sequel:
\begin{equation}\label{expr-A}
A_{i,2n+2} = (\sigma_{2n} \sigma_{2n-1} \cdots \sigma_i)^{-1} \, \sigma_{2n+1}^2 \, 
(\sigma_{2n} \sigma_{2n-1} \cdots \sigma_i). \quad (1 \leq i \leq 2n)
\end{equation}
To derive~\eqref{expr-A} from~\eqref{def-A} use the braid relations or draw a picture.

Recall the special element $\alpha_n = \gamma_n \beta_n \gamma_n^{-1} \beta_n$ given by~\eqref{def-alpha}.
The element $\gamma_n$ belongs to~$P_{2n}$ and $\beta_n$ to $B_{2n}\sigma_{2n+1}^{\pm1}$. 
Since $\sigma_{2n+1}$ commutes with~$B_{2n}$ we have
$\alpha_n\in B_{2n}\sigma_{2n+1}^{\pm2}\cap P_{2n+2}=P_{2n}\sigma_{2n+1}^{\pm2}$.
Consequently, $A_{2n+1,2n+2} = \sigma_{2n+1}^2$ belongs to~$P_{2n} \alpha_n^{\pm 1}$.

Let us deal with $A_{i,2n+2}$ when $1 \leq i \leq 2n$.
Since $\sigma_{2n} \sigma_{2n-1} \cdots \sigma_i$ belongs to~$B_{2n+1}$, we see from~\eqref{expr-A}
and what we established for~$\sigma_{2n+1}^2$  that $A_{i,2n+2}$ is conjugate under~$B_{2n+1}$ of 
an element of $P_{2n} \alpha_n$ or of~$P_{2n}\alpha_n^{-1}$.

So, at the cost of adding~$\alpha_n$, we have reduced $\Ker(\bar{f})$ to its intersection with~$P_{2n+1}$. 
In view of the above considerations, this completes the proof of Part~(a).

(b) Let $N$ be the normal closure of $[\sigma_3, \Delta_3^4]$, $\Delta_5^4  \, \Delta_3^{-16}$ and $\alpha_n$.
We have $N \subset \Ker(f)$: this follows from the centrality of $f(\Delta_3^4) = w_1^4$, 
from Corollary\,\ref{coro-fDelta4}\,(a) applied to $i=2$, and from Proposition\,\ref{prop-ker}.

Let us next remark that $N$ contains the commutator $[\beta, \Delta_3^4]$ for each braid~$\beta$. 
Indeed, since $\Delta_3^4$ is a product of $\sigma_1$ and $\sigma_2$, it commutes with all generators $\sigma_i$ 
and their inverses with $4 \leq i \leq 2n+1$. 
On the other hand $\Delta_3^4$ is central in~$B_3$; therefore it commutes with $\sigma_1$ and $\sigma_2$ and their inverses.
Since by definition $N$ contains $[\sigma_3, \Delta_3^4]$ and we have
$[\sigma_3^{-1}, \Delta_3^4] = \sigma_3^{-1} [\sigma_3, \Delta_3^4]^{-1} \sigma_3$, 
it contains all commutators of the form $[\sigma_i^{\pm 1}, \Delta_3^4]$. Using the commutator identities 
\[
[\beta_1 \beta_2, \Delta_3^4] = \beta_1 [\beta_2, \Delta_3^4] \, \beta_1^{-1} \, [\beta_1, \Delta_3^4]
\]
for $\beta_1, \beta_2 \in B_{2n+2}$, 
we conclude by induction on the length of expression of a braid in the generators~$\sigma_i^{\pm 1}$.

The kernels of~$f$ and of~$\bar{f}$ are connected by the short exact sequence
\begin{equation}\label{ses1}
1 \to \Ker(f) \longrightarrow \Ker(\bar{f}) \overset{f}\longrightarrow \langle w_1^4 \rangle \to 1.
\end{equation}
Indeed, clearly $\Ker(f)$ sits inside~$\Ker(\bar{f})$ as a normal subgroup. 
If $\bar{f}(\beta)= 1$, then $f(\beta)$ belongs to the kernel of~$\pi : \St(C_n,\BZ) \to \Sp_{2n}(\BZ)$, 
which we know to be  infinite cyclic generated by~$w_1^4$. 
The homomorphism $f: \Ker(\bar{f}) \to  \langle w_1^4 \rangle$ is surjective since 
$w_1^4 = f(\Delta_3^4)$ and $\bar{f}(\Delta_3^4) = 1$.
Moreover, since $\langle w_1^4 \rangle \cong \BZ$, the short exact sequence~\eqref{ses1} is split 
with splitting $\langle w_1^4 \rangle \to \Ker(\bar{f})$ given by $w_1^4 \mapsto \Delta_3^4$.

The short exact sequence~\eqref{ses1} induces the quotient short exact sequence
\begin{equation*}
1 \to \Ker(f)/N \longrightarrow \Ker(\bar{f})/N \overset{f}\longrightarrow \langle w_1^4 \rangle \to 1.
\end{equation*}
To conclude that $\Ker(f) = N$, it suffices to check that 
$f: \Ker(\bar{f})/N \to \langle w_1^4 \rangle$ is injective.

Now, by Part\,(a) each element of~$\Ker(\bar{f})$ is a product of conjugates of the braids $\Delta_3^4$, $\Delta_5^4$,
$\alpha_n$ and their inverses. Since $\alpha_n \equiv 1$ and $\Delta_5^4 \equiv (\Delta_3^4)^4$ modulo~$N$,
each element of~$\Ker(\bar{f})$ is equal modulo~$N$ to a product of conjugates of~$\Delta_3^4$ and of its inverse.
Since $[\beta, \Delta_3^4] \in N$ for any braid~$\beta$, we have $\beta \Delta_3^4 \beta^{-1} \equiv \Delta_3^4$ modulo~$N$.
Hence, $\Ker(\bar{f})/N$ is generated by~$\Delta_3^4$ mod~$N$. The image of~$\Delta_3^4$ being $w_1^4$,
the morphism $f: \Ker(\bar{f})/N \to \langle w_1^4 \rangle$ is an isomorphism.
\end{proof}

\begin{rem}\label{rem-BE}
As Benjamin Enriquez pointed out to us for $n=2$, the element~$\alpha_2$ can be rewritten as
$\alpha_2 = \Delta_3^4 \, \Delta_4^{-2} \, \sigma_5^2$.
For $n\geq 3$ let us consider the element
\begin{equation*}
\alpha'_n = \Delta_3^{2n(n-1)} \, \Delta_{2n}^{-2} \, \sigma_{2n+1}^2 \in B_{2n+2} .
\end{equation*}
Using the computations of Section~\ref{ssec-Delta}, it is easy to check that $\alpha'_n$
belongs to the kernel of~$f$. 
By its very definition $\alpha'_n$ is in $P_{2n}\sigma_{2n+1}^{\pm2}$, 
which is an additional feature it shares with~$\alpha_n$.
Reasoning as in the previous proof, we deduce that 
Theorem~\ref{thm-kernel} also holds with~$\alpha'_n$ instead of~$\alpha_n$.
Accordingly, the relation $\alpha_n= 1$ of~\eqref{rel-alpha} can be replaced by 
the relation $\sigma_{2n+1}^2 = \Delta_{2n}^2$.
\end{rem}

\subsection{Restriction to~$B_{2n+1}$}\label{ssec-restrict}

We conclude this section with a result which will be used in Section~\ref{ssec-ker-hat}.
We denote by $f_{2n+1}: B_{2n+1} \to \St(C_n,\BZ)$ (resp.\ $\bar f_{2n+1}: B_{2n+1} \to \Sp_{2n}(\BZ)$)
the restriction of $f$  (resp.\ of~$\bar f$) to~$B_{2n+1}$.

\begin{proposition}\label{prop-ker(f_{2n+1})}
The kernel of $f_{2n+1}: B_{2n+1} \to \St(C_n,\BZ)$ is the normal closure of $[\sigma_3,\Delta_3^4]$ 
and $\Delta_5^4\Delta_3^{-16}$.
\end{proposition}

\begin{proof}
The normal closure~$N'$ of the proposition sits inside $\Ker f_{2n+1}$.
By~\cite{BMP} any element of $\Ker \bar f_{2n+1}$
is a product of conjugates of $\Delta_3^{\pm4}$ and $\Delta_5^{\pm4}$. We
have $\Delta_5^4\equiv \Delta_3^{16}$ modulo~$N'$. 
Reasoning as in the proof of Theorem~\ref{thm-kernel}~(b), we conclude that 
$\Delta_3^4$ is central modulo~$N'$.
As above, there is a short  exact sequence
\begin{equation*}
1 \to \Ker f_{2n+1} \longrightarrow \Ker(\bar{f}_{2n+1}) \longrightarrow \langle w_1^4 \rangle \to 1.
\end{equation*}
The induced quotient short exact sequence is
\begin{equation*}
1 \to \Ker f_{2n+1}/N' \longrightarrow \Ker(\bar{f}_{2n+1})/N' \longrightarrow \langle w_1^4 \rangle \to 1.
\end{equation*}
As in the proof of Theorem \ref{thm-kernel}\,(b), the map onto~$\langle w_1^4 \rangle$ is bijective
and we deduce that the kernel of~$f_{2n+1}$ is~$N'$. 
\end{proof}

\section{Extending $f$ to an epimorphism}\label{sec-epi}

By Theorem~\ref{thm-BtoSt} the homomorphism $f: B_{2n+2} \to \St(C_n,\BZ)$ 
is not surjective when $n\geq 3$. 
We shall now extend $f$ to a surjective homomorphism $\widehat{f} :  \widehat{B}_{2n+2} \to \St(C_n,\BZ)$, 
which we define in the following subsection.
Let us assume that $n$ is a fixed integer $\geq 3$.

\subsection{The Artin groups~$\widehat{B}_k$}\label{ssec-Artin}

For $k\geq 5$ let $\Gamma_k$ be the graph with $k$ vertices labeled $0, 1, \ldots, k-1$ and with unique edges
between the vertices labeled $i$ and $i+1$ for $i \in \{1,\ldots, k-2\}$ plus a unique edge between the vertices labeled $0$
and~$4$.
\begin{equation*}
\nnode1\edge\nnode2\edge\nnode3\edge\vertbar0{\nnode4}\edge\nnode5\halfedge\cdots\halfedge\nnode{$k-2$}\edge\nnode {$k-1$}
\end{equation*}
\begin{center}
\emph{The graph $\Gamma_k$}
\end{center}
Let $\widehat{B}_k$ be the Artin group (also called Artin--Tits group or generalized braid group) 
associated with the graph~$\Gamma_k$; see~\cite{BS, De} or \cite[Sect.~6.6]{KT}. 
It has a presentation with generators $\sigma_0, \sigma_1, \ldots, \sigma_{k-1}$
and with the standard braid relations \eqref{braid2} and \eqref{braid1} between $\sigma_1, \ldots, \sigma_{k-1}$ 
together with the following additional relations involving the generator~$\sigma_0$:
\begin{equation}\label{ref-sigma0}
\sigma_0 \sigma_4 \sigma_0  = \sigma_4 \sigma_0 \sigma_4 
\qquad\text{and} \qquad
\sigma_0 \sigma_i = \sigma_i \sigma_0
\quad (i\neq 4).
\end{equation}

The corresponding Coxeter groups are infinite unless $5 \leq k \leq 7$. 
They are finite of type~$A_5$ if $k=5$, of type~$D_6$ if $k=6$, and of type~$E_7$ if $k=7$
(it is affine of type~$\widetilde{E}_7$ if $k=8$). 

In the same way as there are natural homomorphisms $B_k \to B_{k+1}$, there are natural (injective)
homomorphisms $\widehat{B}_k \to \widehat{B}_{k+1}$ and $j: B_k \to \widehat{B}_k$.
The latter is defined by $j(\sigma_i) = \sigma_i$ for all $i \in \{1,\ldots, k-1\}$.

Here is the reason why we introduce the Artin groups~$\widehat{B}_k$.

\begin{theorem}\label{thm-St-newgen} 
There exists a unique homomorphism 
$\widehat{f}: \widehat{B}_{2n+2} \to \St(C_n,\BZ)$ such that 
$\widehat{f}(\sigma_0) = z_2 $
and $\widehat{f}(\sigma_i) = f(\sigma_i)$ for all $i=1,\ldots, 2n+1$.
The restriction of the homomorphism~$\widehat{f}$ to $\widehat B_{2n+1}$ is surjective.
\end{theorem}

Consequently, the homomorphism $\widehat{f}: \widehat{B}_{2n+2} \to \St(C_n,\BZ)$ is surjective as well.
It extends the non-surjective homomorphism $f: B_{2n+2} \to \St(C_n,\BZ)$
in the sense that $f = \widehat{f} \circ j$.

\begin{proof}
(a) Let us first check that $\widehat{f}$ is well-defined. Since it extends~$f$, we have only to deal with the 
relations~\eqref{ref-sigma0} involving the additional generator~$\sigma_0$.
Now $\widehat{f}(\sigma_0)$ commutes with $\widehat{f}(\sigma_i)  = f(\sigma_i)$ for all $i\neq 4$ 
in view of the defining relations of the Steinberg group (see Section\,\ref{ssec-St}).
The image under~$\widehat{f}$ of the relation $\sigma_0 \sigma_4 \sigma_0  = \sigma_4 \sigma_0 \sigma_4 $
in~\eqref{ref-sigma0} is equivalent to 
$z_2 z_2^{\prime-1} z_2 = z_2^{\prime-1} z_2 z_2^{\prime-1}$, 
which holds in~$\St(C_n,\BZ)$, as explained in the proof of Theorem~\ref{thm-BtoSt}.

(b) By definition of~$\widehat{f}$ it suffices to establish that the image $f(B_{2n+1})$ of~$f$ 
together with~$z_2$ generates the Steinberg group~$\St(C_n,\BZ)$. 

To this end, it is sufficient to prove that the images of~$\bar f(B_{2n+1})$ and~$z_2$ in~$\Sp_{2n}(\BF_2)$
generate the latter finite group.
Indeed, by Theorem~\ref{thm-image} the subgroup~$H$ of $\St(C_n,\BZ)$ generated by~$f(B_{2n+1})$ and~$z_2$ 
contains the kernel $\St(C_n,\BZ)[2]$ of the surjective morphism
$\St(C_n,\BZ)\xrightarrow{\pi} \Sp_{2n}(\BZ) \rightarrow \Sp_{2n}(\BF_2)$.
Hence, if the image of $H$ is the whole group $\Sp_{2n}(\BF_2)$, we obtain the equality $H = \St(C_n,\BZ)$.

Let us denote the images in~$\Sp_{2n}(\BF_2)$ of the generators $x_{i,j}$, $y_{i,j}$, $y'_{i,j}$, $z_i$, $z'_i$ of~$\St(C_n,\BZ)$
by $\bar{x}_{i,j}$, $\bar{y}_{i,j}$, $\bar{y}'_{i,j}$, $\bar{z}_i$, $\bar{z}'_i$ respectively.
Note that the latter generate~$\Sp_{2n}(\BF_2)$ and that each of them is of order~$2$, hence equal to its inverse.

Define $E_n$ to be the set of the images 
of $\{\bar f(\sigma_1),\bar f(\sigma_2),\ldots,\bar f (\sigma_{2n})\}$ in $\Sp_{2n}(\BF_2)$. We have
\[
E_n = \left\{\bar z_1, \bar z'_1, \bar z_1\bar z_2\bar y_{1,2},\bar z'_2, \bar z_2\bar z_3\bar y_{2,3},\bar z'_3,\ldots , \bar z'_n \right\}.
\] 

In order to prove Theorem~\ref{thm-St-newgen}, it is sufficient to check that
$E_n\cup\{\bar z_2\}$ generates $\Sp_{2n}(\BF_2)$. 
We will establish this assertion by induction on~$n$. 

Note that the assertion holds for $n=2$: indeed, $z_2 = \bar{f}(\sigma_5)$ so that 
the set $E_2 \cup \{ \bar z_2 \}$ generates the image of~$\bar{f}$, which by
the surjectivity result of Theorem~\ref{thm-BtoSt} is equal to the whole group~$\Sp_4(\BF_2)$.

We now prove the assertion for $n=3$.

\begin{lemma}\label{lem-sp6}
The set $E_3 \cup \{ \bar z_2 \}$ generates $\Sp_6(\BF_2)$.
\end{lemma}

\begin{proof}
Let $G_3$ be the subgroup of~$\Sp_6(\BF_2)$ generated by 
\[
E_3 \cup \{\bar z_2\} 
=\{ \bar z_1, \bar z'_1, \bar z_1\bar z_2\bar y_{1,2},\bar z'_2, \bar z_2\bar z_3\bar y_{2,3},\bar z'_3,\bar z_2\} .
\]
Obviously, $G_3$ contains $\bar y_{1,2}$ and $\bar z_3 \bar y_{2,3}$.
We have the following relations deduced 
from Equations~\eqref{eq-xxx}, \eqref{eq-xyy}, \eqref{eq-xz} and~\eqref{eq-St9}:
\[
[\bar y_{1,2},\bar z'_2] = \bar x_{1,2}\bar z_1 \, ,\;\; \text{whence} \; \bar x_{1,2}\in G_3 \, ;
\]
\[
[\bar x_{1,2},\bar z_3\bar y_{2,3}]= [\bar x_{1,2},\bar y_{2,3}]=\bar y_{1,3} \, ,\;\; \text{whence} \; \bar y_{1,3}\in G_3 \, ;
\]
\[
[\bar y_{1,3},\bar z'_1]=\bar x_{3,1}\bar z_3 \, ,\;\; \text{whence} \; \bar x_{3,1}\bar z_3\in G_3 \, ;
\]
\[
[\bar y_{1,2},\bar z'_1]=\bar x_{2,1}\bar z_2 \, ,\;\; \text{whence} \; \bar x_{2,1} \in G_3 \, ;
\]
\[
[\bar y_{1,3},\bar z'_3]=\bar x_{1,3} \bar z_1 \, ,\;\; \text{whence} \; \bar x_{1,3} \in G_3 \, ;
\]
\[
[\bar x_{3,1}\bar z_3,\bar x_{1,2}]=[\bar x_{3,1},\bar x_{1,2}]=\bar x_{3,2} \, ,\;\; \text{whence} \; \bar x_{3,2} \in G_3 \, ;
\]
\[
[\bar x_{3,2},\bar x_{2,1}]=\bar x_{3,1} \, ,\;\; \text{whence} \; \bar x_{3,1} \in G_3 \, .
\]
Since $\bar x_{3,1}$ and $\bar x_{3,1}\bar z_3$ belong to~$G_3$, so does~$\bar z_3$.
Now we know that all $\bar y_{i,j}$, $\bar z_i$ and $\bar z'_i$ with $i,j \in \{1,2,3\}$ and $i\neq j$ belong to~$G_3$. 
As follows from a remark in Section~\ref{ssec-St}, these elements generate the whole group~$\Sp_6(\BF_2)$.
\end{proof}

We resume Part~(b) of the proof of Theorem~\ref{thm-St-newgen}.
Assume that the assertion above holds for $n\geq 3$ and let us prove it for~$n+1$.

Let $G_{n+1}$ be the subgroup of $\Sp_{2n+2}(\BF_2)$ generated by $E_{n+1}\cup \{\bar z_2\}$.
By the induction hypothesis, since $E_n$ is a subset of~$E_{n+1}$, 
the group~$G_{n+1}$ contains $\Sp_{2n}(\BF_2)$, viewed as the group of matrices in~$\Sp_{2n+2}(\BF_2)$
with entries equal to~$0$ for the indices $(i,n+1)$ and $(n+1,i)$ with $i\neq n+1$ and 
for the indices $(i,2n+2)$ and $(2n+2,i)$ with $i\neq 2n+2$. 
In particular, $G_{n+1}$ contains~$\bar z_3$ by Lemma~\ref{lem-sp6}. 

Now consider the subgroup of~$G_{n+1}$ generated by
\[
\{\bar z_2, \bar z'_2, \bar z_2\bar z_3\bar y_{2,3},\bar z'_3, \bar z_3\bar z_4\bar y_{3,4},\bar z'_4,\ldots , \bar z'_{n+1}\}\cup\{\bar z_3\}.
\] 
Since $\bar z_3\in G_{n+1}$, this subgroup is a subgroup of~$G_{n+1}$. 

By the induction hypothesis applied to all subscripts increased by~$1$, this subgroup
is isomorphic to the symplectic group~$\Sp_{2n}(\BF_2)$, now viewed
as the group of matrices in~$\Sp_{2n+2}(\BF_2)$
with entries equal to~$0$ for the indices $(i,1)$ and $(1,i)$ with $i\neq1$
and for the indices $(i,n+2)$ and $(n+2,i)$ with $i\neq n+2$. 

It follows that all generators $\bar{x}_{i,j}$, $\bar{y}_{i,j}$, $\bar{y}'_{i,j}$, $\bar{z}_i$, $\bar{z}'_i$ of~$\Sp_{2n+2}(\BF_2)$
belong to~$G_{n+1}$, except possibly $\bar{x}_{1,n+1}$, $\bar{x}_{n+1,1}$, $\bar{y}_{1,n+1}$ and $\bar{y}'_{1,n+1}$.
But the latter also belong to~$G_{n+1}$ in view of the commutator relations
\begin{equation*}
[\bar x_{1,n},\bar x_{n,n+1}] = \bar x_{1,n+1}\, , \quad
[\bar x_{n+1,n},\bar x_{n,1}] = \bar x_{n+1,1}\, ,
\end{equation*}
\begin{equation*}
[\bar x_{1,n},\bar y_{n,n+1}] = \bar y_{1,n+1} \, , \quad
[\bar x_{n,1},\bar y'_{n,n+1}] = \bar y'_{1,n+1}\, ,
\end{equation*}
which follow from Relations~\eqref{eq-xxx}--\eqref{eq-xyprime}.
\end{proof}

\subsection{Elements of the kernel of~$\widehat f$}\label{ssec-ker-hat}

Let $\widehat f_{2n+1}:\widehat B_{2n+1}\rightarrow \St(C_n,\BZ)$ be the restriction of $\widehat f$ to~$\widehat B_{2n+1}$.
Recall the restrictions $f_{2n+1}: B_{2n+1} \to \St(C_n,\BZ)$ and $\bar f_{2n+1}: B_{2n+1} \to \Sp_{2n}(\BZ)$ 
defined in Section~\ref{ssec-restrict}.

We consider the element 
\begin{equation}\label{def-alpha0}
\alpha_0=(\sigma_1\sigma_2\sigma_1)^2(\sigma_1\sigma_3^{-1}\sigma_0)
(\sigma_1\sigma_2\sigma_1)^{-2}(\sigma_1\sigma_3^{-1}\sigma_0)
\end{equation}
of the Artin group $\widehat{B}_{2n+1}$ of type $\Gamma_{2n+1}$.
This element lies in the kernel of $\widehat f_{2n+1}$ since
$\widehat f_{2n+1}(\alpha_0) = w_1^2 y_{1,2} w_1^{-2} y_{1,2} = y_{1,2}^{-2} y_{1,2} = 1$.

Note that mapping $\sigma_0$ to~$\sigma_5$ and $\sigma_i$ to~$\sigma_i$ for $i=1,\ldots, 4$, 
we obtain an isomorphism~$j$ from the subgroup~$\widehat B_5$ of~$\widehat B_{2n+1}$
generated by $\{\sigma_1,\sigma_2,\sigma_3,\sigma_4,\sigma_0\}$ to the standard braid group~$B_6$.
The composed morphism
\[
\langle \sigma_1,\sigma_2,\sigma_3,\sigma_4,\sigma_0 \rangle \overset{j}\longrightarrow
B_6 \overset{f}\longrightarrow \St(C_2,\BZ) 
\]
is equal to the restriction of~$\widehat f_{2n+1}$ and $j$~maps~$\alpha_0$ 
to~$\alpha_2$, which  by Theorem~\ref{prop-ker} lies in the kernel of $f: B_6 \to \St(C_2,\BZ)$. 
We thus recover the fact that $\alpha_0$~belongs to~$\Ker (\widehat f_{2n+1})$.

\begin{theorem}\label{th-ker(hat f_{2n+1})}
The kernel of the restriction of $\widehat f_{2n+1}$ to the pure Artin group~$\widehat P_{2n+1}$ of type~$\Gamma_{2n+1}$ 
is the normal closure of $\alpha_0$, $[\sigma_3,\Delta_3^4]$ and $\Delta_5^4\Delta_3^{-16}$ in~$\widehat B_{2n+1}$.
\end{theorem}

\begin{proof}
It follows from Proposition~\ref{prop-ker(f_{2n+1})} that 
\[
\Ker(\widehat f_{2n+1})\cap P_{2n+1} = \Ker(f_{2n+1})\cap P_{2n+1}
\] 
is the normal closure 
of $[\sigma_3,\Delta_3^4]$ and $\Delta_5^4\Delta_3^{-16}$ in~$B_{2n+1}$.

To determine the kernel in~$\widehat P_{2n+1}$ we apply~\cite[Corollary~8]{DG} (and its proof). 
This result states that if $\Gamma$ is a Coxeter graph and $i$ a vertex of~$\Gamma$,
the pure Artin group of type $\Gamma$ is the semi-direct product of a normal subgroup 
generated by conjugates of the squares $\sigma_j^2$ with $j\in \Gamma$
and of the pure Artin group of type~$\Gamma\setminus \{i\}$. 

We apply this to $\Gamma=\Gamma_k$ and $i=0$. All generators $\sigma_j$
are conjugate in~$\widehat B_{2n+1}$; in particular their squares are conjugate to~$\sigma_0^2$.
Consequently, $\widehat P_{2n+1}$ is the semi-direct product of a normal subgroup
generated by elements conjugate to~$\sigma_0^2$ in~$\widehat B_{2n+1}$ and of~$P_{2n+1}$. 
Since $\alpha_0$ lies in $P_4 \sigma_0^2\subset P_{2n+1} \sigma_0^2$ 
(see the proof of Theorem~\ref{thm-kernel})
and $\alpha_0 \in \Ker\widehat f_{2n+1}$, we deduce the theorem.
\end{proof}

\begin{rem}
By Theorem~\ref{thm-St-newgen} the morphism $\widehat f_{2n+1}: \widehat{B}_{2n+1} \to \St(C_n,\BZ)$ is surjective.
Composing it with the natural surjections $\St(C_n,\BZ) \to \Sp_{2n}(\BZ) \to \Sp_{2n}(\BF_2)$, we obtain 
a surjective morphism $\widehat{B}_{2n+1} \to\Sp_{2n}(\BF_2)$.
The latter epimorphism factors through $W(\Gamma_{2n+1})\rightarrow \Sp_{2n}(\BF_2)$, 
where $W(\Gamma_{2n+1})$ is the Coxeter group associated with the graph~$\Gamma_{2n+1}$,
since the image of each generator~$\sigma_i$ has order~$2$ for $i=0,\ldots,2n$.
\end{rem}

\begin{rem}
We can say a little more about the kernel of $\widehat{f}_{2n+1}$ when $n=3$. 
In this case the Coxeter group $W(\Gamma_7)$ is of type~$E_7$ and its center has order~$2$. 
Let~$w_0$ be its non-trivial central element. Since the center of $\Sp_6(\BF_2)$ is trivial, 
the element~$w_0$ has to be in the kernel of the above surjective morphism
$W(\Gamma_7)\rightarrow \Sp_6(\BF_2)$. 
Since the order of $W(\Gamma_7)$ is twice
the order of $\Sp_6(\BF_2)$, this kernel is exactly $\{1,w_0\}$.
We deduce that any element of $\Ker \widehat f_7$ is either in the pure
Artin group $\widehat P_7$ of type $\Gamma_7$, or in the coset $\bw_0
\widehat P_7$, where $\bw_0$ is a fixed preimage of $w_0$ in $\widehat B_7$. 

Since we know $\Ker\widehat{f}_7 \cap \widehat P_7$ by Theorem~\ref{th-ker(hat f_{2n+1})}, 
it remains to determine the intersection of~$\Ker\widehat{f}_7$ with the coset $\bw_0 \widehat P_7$.
A computation shows that the image of $\bw_0$ in $\Sp_6(\BZ)$ is trivial.
Thus $\widehat {f}_7(\bw_0)=w_1^{4k} \in \St(C_3,\BZ)$ for some integer~$k$.
Consequently, $\bw_0 \Delta_3^{-4k}$ belongs to $\Ker \widehat{f}_7$
and the full kernel $\Ker \widehat{f}_7: \widehat{B}_7 \to \St(C_3,\BZ)$ is the normal closure of 
$\alpha_0$, $[\sigma_3,\Delta_3^4]$, $\Delta_5^4\Delta_3^{-16}$ and $\bw_0 \Delta_3^{-4k}$. 
We have not been able to determine the exponent~$k$.
\end{rem}

\section*{Acknowledgement}
We are grateful to Christian Blanchet for having drawn our attention to References \cite{Ar, BMP}
and helped us understand the connections with the hyperelliptic Torelli groups.
We also thank Tara Brendle for reading a preliminary draft and 
an anonymous referee for bringing the preprint~\cite{BPS} to our attention.
We are also very grateful to Benjamin Enriquez for his observation on the braid~$\alpha_2$,
which led to Remark~\ref{rem-BE}.


\end{document}